\documentclass{amsart}

\usepackage{amsthm}
\usepackage{amsmath}
\usepackage{amssymb}
\usepackage{IMjournal}

\usepackage{graphicx}
\usepackage{multirow}
\usepackage{url}
\usepackage{siunitx}

\newtheorem{theorem}{Theorem}
\newtheorem{lemma}[theorem]{Lemma}

\usepackage{algorithmic}
\usepackage{color}
\usepackage{float}
\definecolor{note_fontcolor}{rgb}{0.800781, 0.800781, 0.800781}
\usepackage{graphicx}
\usepackage{subfig}

\providecommand{\tabularnewline}{\\}
\floatstyle{ruled}
\newfloat{algorithm}{tbp}{loa}
\providecommand{\algorithmname}{Algorithm}
\floatname{algorithm}{\protect\algorithmname}

\usepackage{bm}

\usepackage{nicefrac}    % compact symbols for 1/2, etc.
\usepackage{booktabs}    % professional-quality tables

\let\originalleft\left
\let\originalright\right
\renewcommand{\left}{\mathopen{}\mathclose\bgroup\originalleft}
\renewcommand{\right}{\aftergroup\egroup\originalright}

\theoremstyle{definition}

\newcommand{\norm}[1]{\|#1\|}

\newcommand{\bpm}{\begin{pmatrix}}
\newcommand{\epm}{\end{pmatrix}}

\newcommand{\vphi}{\varphi}

\newcommand{\diag}{\operatorname{diag}}

\begin{document}

\numberwithin{equation}{section}
\title{Solvability of the Power Flow Problem \\in DC Overhead Wire Circuit
Modeling}

%\author{\|Jakub |\v{S}ev\v{c}\'ik|, Pilsen,
%        \|Luk\'a\v{s} |Adam|, Prague,
%        \|Jan |P\v{r}ikryl|, Pilsen,
%        \|V\'aclav |\v{S}m\'idl|, Pilsen,}
\author{\|Jakub |\v{S}ev\v{c}\'ik|,
        \|Luk\'a\v{s} |Adam|,
        \|Jan |P\v{r}ikryl|,
        \|V\'aclav |\v{S}m\'idl|,}
%\author{\|FirstAuthorFirstName |FirstAuthorSurname|, FirstAuthorTown,
%        \|SecondAuthorFirstName |SecondAuthorSurname|, SecondAuthorTown, ...}

\rec {September 30, 2020}

%\dedicatory{Cordially dedicated to ...}

\abstract 
Proper traffic simulation of electric vehicles, which draw energy from overhead wires, requires adequate modeling of traction infrastructure. Such vehicles include trains, trams or trolleybuses. Since the requested power demands depend on a traffic situation, the overhead wire DC electrical circuit is associated with a non-linear power flow problem. Although the Newton-Raphson method is well-known and widely accepted for seeking its solution, the existence of such a solution is not guaranteed. Particularly in situations where the vehicle power demands are too high (during acceleration), the solution of the studied problem may not exist. To deal with such cases, we introduce a numerical method which seeks maximal suppliable power demands for which the solution exists. This corresponds to introducing a scaling parameter to reduce the demanded power. The interpretation of the scaling parameter is the amount of energy which is absent in the system, and which needs to be provided by external sources such as on-board batteries. We propose an efficient two-stage algorithm to find the optimal scaling parameter and the resulting potentials in the overhead wire network. We perform a comparison with a naive approach and present a real-world simulation of part of the Pilsen city in the Czech Republic. These simulations are performed in the traffic micro-simulator SUMO, a~popular open-source traffic simulation platform.
\endabstract

\keywords
   power flow problem; Newton-Raphson method; solvability; scaling parameter
\endkeywords

\subjclass
%%%%%
%%%Mathematics Subject Classification 2020
%%%%%
49Mxx, %58C15, 
65J15, %68Q06, 78M99, 
94C60,
90C30
\endsubjclass

\thanks
   This work was supported by the Ministry of Education, Youth and Sports of the Czech Republic under the project OP VVV Electrical Engineering Technologies with High-Level of Embedded Intelligence, CZ.02.1.01/0.0/0.0/18\_069/0009855, project OP VVV Research Center for Informatics, CZ.02.1.01/0.0./0.0./16\_019/0000765, and by UWB Student Grant Project no.~SGS-2018-009.
\endthanks

\section{Introduction \label{sec:Introduction}}

Electrification of transport belongs to one of the key targets of
the automotive industry today. The electrification of public transport
road vehicles in urban areas is feasible and well-used for decades
employing trolleybuses or recently hybrid trolleybuses (i.e. dynamically
charging e-buses with a battery pack on the board). A replacement
of classic buses (with a combustion engine) with (hybrid) trolleybuses
is, nonetheless, hardly possible without an appropriate adjustment
and dimensioning of the necessary traction infrastructure. For this
purpose, a simultaneous simulation of the power network and traffic
conditions needs to be used to identify weaknesses of the proposed solution
\cite{sevcik2019sumo}.

The trolleybus overhead wire network is typically a direct current
(DC) electric circuit, where traction substations supply electric
energy. The connected trolleybuses represent power loads. The steady
state analysis of such circuits enables monitoring of voltage drops,
undesirable over-currents, power losses and therefore effective dimensioning
of suggested overhead wire networks in urban areas. 
In this manuscript, we consider the well-known DC power flow (PF) problem, where electric traction substations are modeled as constant voltage
sources, the resistance of overhead wires is replaced by ideal resistor elements with the resistance linearly dependent on the distance between nodes; and trolleybuses are substituted by a source with a defined power load (Figure \ref{fig:dc_network}). %\textcolor{red}{\textbf{[JS suggest to place this sentence here.]} and trolleybuses are substituted by a defined power load or a power source (power load with negative sign) if they are regenerating energy during breaking (Figure \ref{fig:dc_network}).} % We would like to stress that this approach allows for the trolleybus regenerating energy during breaking by considering it as power source (power load with negative sign).}

%In the substitution 
%circuit, electric traction substation can be modeled as constant voltage
%source, the resistance of overhead wires is replaced by ideal resistor
%element with the resistance linearly dependent on the distance between
%nodes, and a trolleybus is substituted by a source with the defined
%power load (Figure \ref{fig:dc_network}). Note, that trolleybus regenerating
%energy during breaking can be similarly modeled as power source (i.e.
%power load with negative sign). These constant power loads/sources
%lead to a non-linear set of algebraic equations, that is needed to
%solve to obtain unknown values of node voltages and edge currents.
%This problem is known as DC power flow (PF) problem in the literature
%and is now briefly reviewed.

\begin{figure}
\centering
\fontsize{9pt}{12pt}\selectfont
\def\svgwidth{0.75\columnwidth}
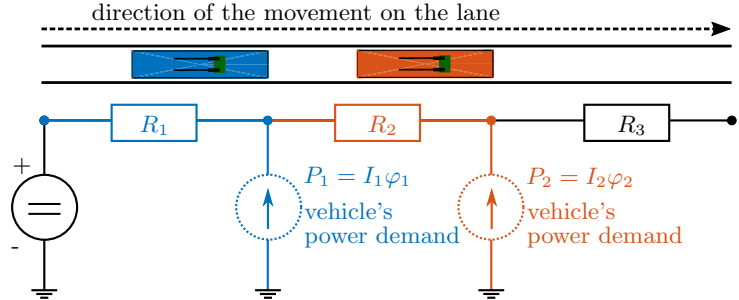
\caption{Simple DC network in ladder-form with moving trolleybuses \label{fig:dc_network}}
%
%\begin{centering}
%\includegraphics[width=0.8\linewidth]{../prezentace_vs/obr/scheme_circuit2}
%\par\end{centering}
%\caption{Simple DC network in ladder-form with moving trolleybuses \label{fig:dc_network}}
\end{figure}

Although there exists a rich literature on the alternating current
(AC) power flow problem \cite{Taheri_PFSolver}, papers on the DC PF problem
were limited in the past. In the majority of cases, they dealt with the DC PF
problem only from application point of view \cite{salihimpact,Ku_SolutionOfDCPFforNongrounded,zhou1994SimulationOfDCPowerDistrubutionSystems,jayarathna2014LoadFlowAnslysis}
 and without mentioning the solvability of the problem or uniqueness
of the solution. Regardless, the majority of these works succeeded with
utilizing Gauss-Seidel or Newton-Raphson methods to solve a well-defined
system (i.e.\ system with well-defined operative conditions given by
suitable values of variables) of DC PF non-linear equations. In recent
years, the DC PF problem receives more attention in literature, since
it is connected to low voltage DC grids,
an appealing concept in the field of smart grids and microgrids.

Garces proves uniqueness \cite{Garces_uniquenessOfPFSolution} of
the solution of the DC PF problem (using Banach fixed-point theorem),
and even convergence of Gauss-Seidel and Newton-Raphson methods \cite{Garces_ConvergendeNRMethodForDCPF,Garces_uniquenessOfPFSolution}, both under a set of reasonable and not much restricting
assumptions but enforcing sufficiently low power demands. Further,
Taylor's series expansion was used to linearize the DC
PF problem in \cite{montoya2018linear}. Taheri and Kekatos \cite{Taheri_PFSolver}
proposed three various approaches to solve the DC PF problem assuming
bounded power demands, and suggested a decision tree to select the
proper method with guaranteed convergence. The DC PF problem was
also reformulated as an optimization task \cite{tan2012DCOptimalPF,garces2019potentialFunction} and its solvability was discussed.

However, the existence of the solution of the DC PF problem generally
heavily depends on the power demands \cite{Dorfler_ElectricalNetworksAndAlgebraicGraphTheory}. Even for a primitive circuit (Figure \ref{fig:simple-circuit}), none, one, or two real-valued solutions of the DC PF problem may exist, dependent on the value of the power demand. It can be shown,
that the unknown potential $\varphi_{2}$ (the notation is borrowed
from Figure \ref{fig:simple-circuit}) is the solution of quadratic
equation ${\varphi_{2}^{2}-V\varphi_{2}+PR=0}$. Therefore, there exists a critical
value of power demand $P_\text{crit}={V^{2}}/{4R}$ and the
existence of the DC PF problem's solution depends on the ratio between
the demanded power and the critical power value. If the demanded power
is higher than the critical power value, the DC PF problem has no
solution. In such situations, the demanded power cannot be
supplied in real-world conditions due to physical restrictions.
% or due to over-current protection of traction substation.

\begin{figure}
\begin{centering}
\centering
\fontsize{9pt}{14pt}\selectfont
\def\svgwidth{0.37\columnwidth}
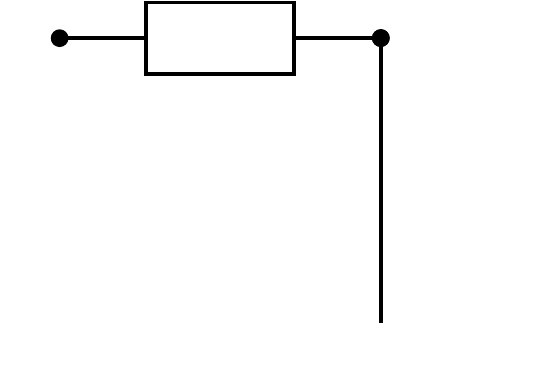
\hspace{1.5cm}
\fontsize{9pt}{14pt}\selectfont
\def\svgwidth{0.33\columnwidth}
%% Creator: Inkscape inkscape 0.92.3, www.inkscape.org
%% PDF/EPS/PS + LaTeX output extension by Johan Engelen, 2010
%% Accompanies image file 'solution_voltage2.pdf' (pdf, eps, ps)
%%
%% To include the image in your LaTeX document, write
%%   \input{<filename>.pdf_tex}
%%  instead of
%%   \includegraphics{<filename>.pdf}
%% To scale the image, write
%%   \def\svgwidth{<desired width>}
%%   \input{<filename>.pdf_tex}
%%  instead of
%%   \includegraphics[width=<desired width>]{<filename>.pdf}
%%
%% Images with a different path to the parent latex file can
%% be accessed with the `import' package (which may need to be
%% installed) using
%%   \usepackage{import}
%% in the preamble, and then including the image with
%%   \import{<path to file>}{<filename>.pdf_tex}
%% Alternatively, one can specify
%%   \graphicspath{{<path to file>/}}
%% 
%% For more information, please see info/svg-inkscape on CTAN:
%%   http://tug.ctan.org/tex-archive/info/svg-inkscape
%%
\begingroup%
  \makeatletter%
  \providecommand\color[2][]{%
    \errmessage{(Inkscape) Color is used for the text in Inkscape, but the package 'color.sty' is not loaded}%
    \renewcommand\color[2][]{}%
  }%
  \providecommand\transparent[1]{%
    \errmessage{(Inkscape) Transparency is used (non-zero) for the text in Inkscape, but the package 'transparent.sty' is not loaded}%
    \renewcommand\transparent[1]{}%
  }%
  \providecommand\rotatebox[2]{#2}%
  \newcommand*\fsize{\dimexpr\f@size pt\relax}%
  \newcommand*\lineheight[1]{\fontsize{\fsize}{#1\fsize}\selectfont}%
  \ifx\svgwidth\undefined%
    \setlength{\unitlength}{382.5bp}%
    \ifx\svgscale\undefined%
      \relax%
    \else%
      \setlength{\unitlength}{\unitlength * \real{\svgscale}}%
    \fi%
  \else%
    \setlength{\unitlength}{\svgwidth}%
  \fi%
  \global\let\svgwidth\undefined%
  \global\let\svgscale\undefined%
  \makeatother%
  \begin{picture}(1,0.75178259)%
    \lineheight{1}%
    \setlength\tabcolsep{0pt}%
    \put(0,0){\includegraphics[width=\unitlength,page=1]{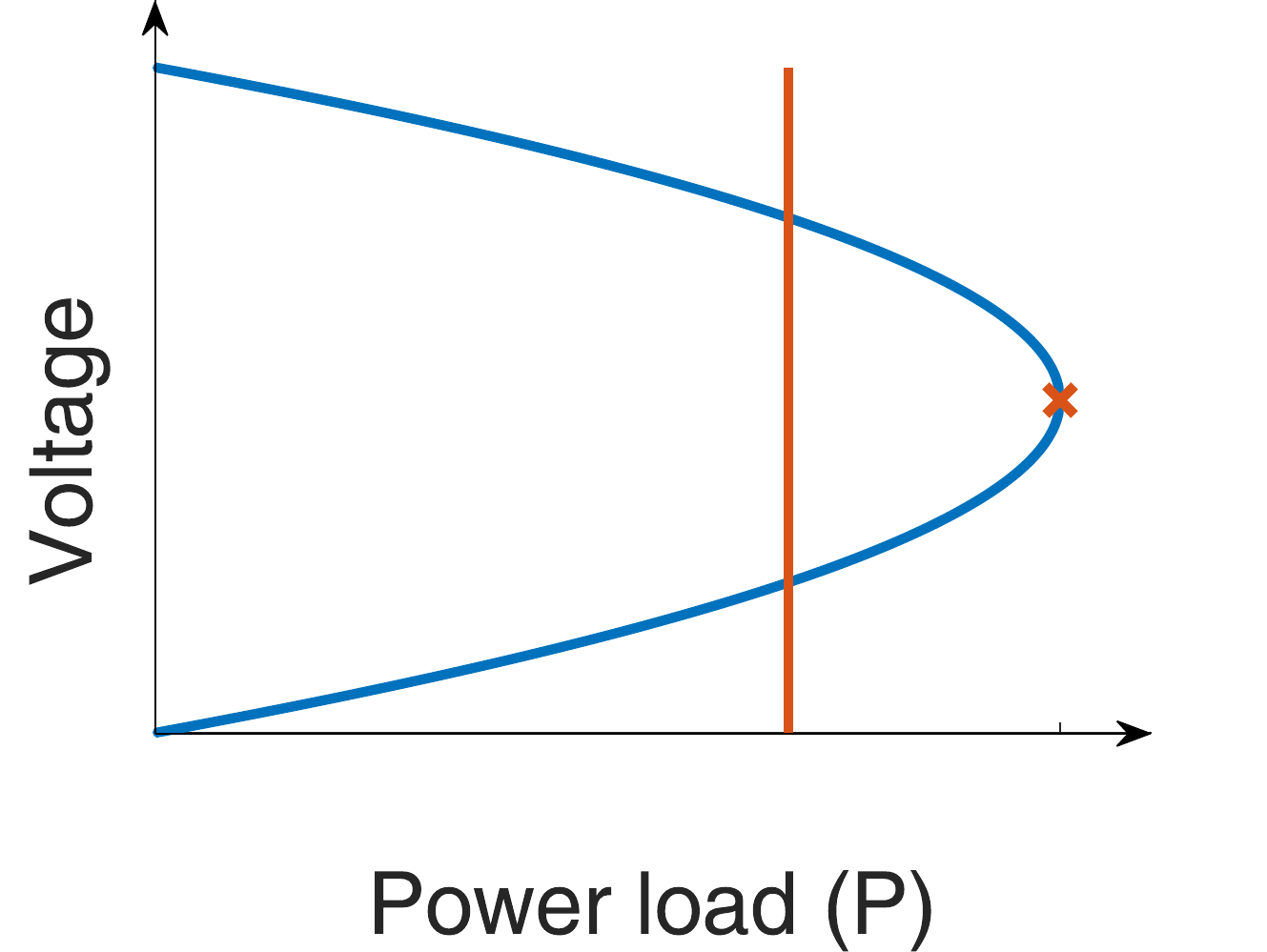}}%
    \put(0.02626888,0.67557462){\color[rgb]{0,0,0}\makebox(0,0)[lt]{\lineheight{1.25}\smash{\begin{tabular}[t]{l}$\varphi_2$\end{tabular}}}}%
    \put(0.56603931,0.10035279){\color[rgb]{0,0,0}\makebox(0,0)[lt]{\lineheight{1.25}\smash{\begin{tabular}[t]{l}$P_\text{load}$\end{tabular}}}}%
    \put(0.78851108,0.10035279){\color[rgb]{0,0,0}\makebox(0,0)[lt]{\lineheight{1.25}\smash{\begin{tabular}[t]{l}$P_\text{crit}$\end{tabular}}}}%
  \end{picture}%
\endgroup%

\par\end{centering}
\caption{Simple DC circuit (left) with unknown physical values emphasized with
red color, and its solution (right).\label{fig:simple-circuit}}
\end{figure}

In this manuscript, we test the limitation of overhead wire infrastructure
and simulate scenarios close to physical limits and even behind them.
%from overhead wire network point of view. 
For this reason, a method solving the DC PF problem with large power loads is
necessary.
%In this contribution, we propose a novel approach to handle a DC PF
%problem with unbounded power loads.
We introduce a scaling parameter
that uniformly decreases the original demanded power values. We then propose a method to find a critical value of power demand where the overhead wire circuit is already
solvable. Such decreased power values can be seen as the maximal suppliable
power loads. 
%Moreover, we propose a strategy to find the optimal value
%of the scaling parameter, that can be interpreted as an overload rate
%of the overhead wire circuit.
This corresponds to finding a maximal value of the scaling parameter. We propose a two-stage strategy, where the first phase searches for the scaling parameter while the second stage verifies the solvability of the overhead wire circuit. Since the scaling parameter is one-dimensional, this allows us to use a combination of a bisection method in the first stage and Newton-Raphson method in the second stage. 

This paper is organized as follows. We revisit the mathematical
formulation of the DC PF problem in Section \ref{sec:Mathematical-Formulation},
and we introduce a method for seeking maximal suppliable power demands
to ensure the existence of the solution in Section \ref{sec:Scaling-Parameter}.
The theoretical analysis of the proposed approach is carried out in Section \ref{sec:theoreticalAnalysis}, where we show that our method is convergent under mild assumptions. 
A solution procedure employing introduced theoretical results and an efficient algorithm to find the solution with lower time requirements are proposed in Section \ref{sec:Proposed-Algorithm}. Finally, a numerical insight into the problem, a comparison with standard non-linear solver in Matlab, and a simulation of a city trolleybus network are presented in Section~ \ref{sec:Numerical-results}.

\section{Mathematical Formulation \label{sec:Mathematical-Formulation}}
The DC PF problem corresponds to finding unknown potentials of nodes and
branch currents in an electric circuit with defined voltage
and power loads/sources. There are various methods to formulate the corresponding
set of equations. Since the fundamental electrical laws (Kirchhoff's
circuit laws, Ohm's law, Power law) need to be always employed, the
resulting formulations are naturally equivalent. We use modified nodal
analysis \cite{ho1975modifiedNodalAnalysis} to form the
system of equations for the electric circuit. Hence, given the specific
application in overhead wire circuit modeling, the system (a connected
electric circuit with voltage sources and power loads) can be described
by three types of equations. Each equation describes currents going
through a selected node with unknown potentials or assigns known voltage sources to  adjacent nodes.

If node $i$ is not adjacent to any voltage source or power loads (i.e.
the adjacent nodes are connected only through resistor elements),
the corresponding equation reads
\begin{align}
\text{\ensuremath{\sum_{j\in N(i)}\frac{1}{R_{ij}}\left(\varphi_{i}-\varphi_{j}\right)}} & =0, & i & \in I,\label{eq:DCPF_resistor_nodes}
\end{align}
where $R_{ij}$ is the resistance between the $i$-th and $j$-th nodes,
$\varphi_{i}$ is the potential of the $i$-th node, $N(i)$ is the
set of adjacent nodes to the $i$-th nodes through resistor elements, and
$I$ is an index set of nodes connected only by resistor elements.

If node  $i$  is adjacent to a power load, the equation obtained using
modified nodal analysis is similar to (\ref{eq:DCPF_resistor_nodes}),
but with a non-zero right-hand side describing current going to/from the
power load/source
\begin{align}
\ensuremath{\sum_{j\in N(i)}\frac{1}{R_{ij}}\left(\varphi_{i}-\varphi_{j}\right)} & =\frac{P_{i}}{\varphi_{i}}, & i & \in J,\label{eq:DCPF_power_loads}
\end{align}
where $P_{i}$ is the demanded power by the power load adjacent to the
$i$-th node and $J$ is an index set of nodes adjacent to any power
load.

The rest of nodes are connected to voltage sources, so their potentials
are defined as
\begin{align}
\varphi_{i} & =U_{i}, & i & \in I_{0},\label{eq:DCPF_voltage_sources}
\end{align}
where $U_{i}$ is a known voltage level of voltage source at the $i$-th node
and $I_{0}$ is an index set of nodes adjacent to any voltage source.

The DC PF problem then amounts to finding the unknown potentials $\varphi_{i}, i\in\left\{ I,J,I_{0}\right\} $.
If we solve equations (\ref{eq:DCPF_resistor_nodes})--(\ref{eq:DCPF_voltage_sources})
for unknown potentials, currents through the circuit can be %then always
found by Ohm's and Power laws. 
%Note that the ground node
%(i.e. the reference node in modified nodal analysis) $\varphi_{0}$
%does not belong to any mentioned sets, i.e. $\varphi_{0}\notin \left\{\varphi_i|  i \in I\cup J\cup I_{0}\right\}$,
%and these sets are mutually disjoint and complete subsets of the set
%of all nodes $\mathcal{N}=\left\{\varphi_i|  i \in I\cup J\cup I_{0}\right\}\cup\varphi_0 $.

System (\ref{eq:DCPF_resistor_nodes})--(\ref{eq:DCPF_voltage_sources})
can be rewritten in a compact form introducing square matrix $A$
and a vector $\boldsymbol{b}(\boldsymbol{\varphi})$ 
\begin{equation}
A\boldsymbol{\varphi}=\boldsymbol{b}(\boldsymbol{\varphi})\equiv\left[\begin{array}{c}
\boldsymbol{0}\\
\frac{\boldsymbol{P}}{\boldsymbol{\varphi}_{J}}\\
\boldsymbol{U}
\end{array}\right],\label{eq:DCPF_matrix_form}
\end{equation}
where $\boldsymbol{\varphi}$ is an ordered vector of unknown potentials
$\varphi_{i}, i\in\left\{ I,J,I_{0}\right\} $, $\boldsymbol{0}$ is
a zero vector, $\boldsymbol{P}$ is a vector of power demands with
components $P_{i}, i\in J$, $\boldsymbol{\varphi}_{J}$
is a vector with components $\varphi_{i}, i\in J$, $\frac{\boldsymbol{P}}{\boldsymbol{\varphi}_{J}}$
is the element-wise division and the vector $\boldsymbol{U}$ consists
of components $U_{i}, i\in I_{0}$.
%Matrix $A$ is a square matrix of size $n\times n$, where $n=\left|\left\{ I,J,I_{0}\right\} \right|$.

\section{Reduction of Power Demands to Ensure Solvability\label{sec:Scaling-Parameter}}

To find a solution  of equation
(\ref{eq:DCPF_matrix_form}), we can employ numerical methods as Gauss-Seidel
method or Newton-Raphson method. %, if the solution exists. 
 However, as it has been discussed in Section \ref{sec:Introduction}, the existence
of the solution is not guaranteed in all situations. Especially if the
power demands are too high, the solution of the problem
does not exist. In such a situation, the power demands of trolleybuses
cannot be fully satisfied due to physical limits or due to over-current
protections of traction substation.

To find the maximal suppliable values of power and to guarantee the
existence of solution, we adjust (\ref{eq:DCPF_power_loads}) by introducing
a vector $\boldsymbol{\alpha}$ of scaling parameters $\alpha_i$ by
\begin{align}
\ensuremath{\sum_{j\in N(i)}\frac{1}{R_{ij}}\left(\varphi_{i}-\varphi_{j}\right)} & =\frac{\alpha_i P_{i}}{\varphi_{i}}, & i & \in J.
\end{align}
Equations (\ref{eq:DCPF_resistor_nodes}) and (\ref{eq:DCPF_voltage_sources})
are without any change. To simplify matters, parameter vector $\boldsymbol{\alpha}$ is considered as all-ones vector multiplied by a scalar value, i.e. $\boldsymbol{\alpha}=\alpha\boldsymbol{1}$. While this formulation does not fully correspond to the physical equilibrium of the system, the scalar approximation makes the system computationally tractable.
Note that $\alpha=1$ corresponds to the original
problem and $\alpha=0$ to the situation when no demanded power is supplied. Again,
the problem can be reformulated into the compact form similar to (\ref{eq:DCPF_matrix_form}),
\begin{equation}
f(\boldsymbol{\varphi}(\alpha),\alpha)=A\boldsymbol{\varphi}(\alpha)-\boldsymbol{b}(\boldsymbol{\varphi}(\alpha),\alpha)=A\boldsymbol{\varphi}(\alpha)-\left[\begin{array}{c}
\boldsymbol{0}\\
\frac{\alpha\boldsymbol{P}}{\boldsymbol{\varphi}(\alpha)_{J}}\\
\boldsymbol{U}
\end{array}\right]=0,\label{eq:scaling_matrix_form-1}
\end{equation}
where we denote the dependency of $\boldsymbol{\varphi}$ on the scalar parameter
$\alpha$ by $\boldsymbol{\varphi}(\alpha).$

The value of $\alpha$ can be seen as an overload rate of the investigated
electric circuit, or in other words as a power demand satisfaction rate. If the equation (\ref{eq:scaling_matrix_form-1})
has the solution for $\alpha=1$ (i.e.\ the original problem), the investigated
circuit is not overloaded and all power loads are fully supplied by
the circuit. For $\alpha=0$, i.e.\ the power demands are completely disregarded,  (\ref{eq:scaling_matrix_form-1}) has always a trivial solution
with zero currents and with nominal voltage of traction substation
at all nodes in the circuit.

The idea of the proposed method is to find some $\alpha_{0}\in\left[0,1\right]$
such that (\ref{eq:scaling_matrix_form-1}) has a solution for $\alpha\in\left[0,\alpha_{0}\right]$
and does not have a solution for $\alpha\in\left(\alpha_{0},1\right]$.
Since the scaling parameter evenly decreases values of the demanded
power, $\alpha_{0}$ can be also defined as a maximal $\alpha\in\left[0,1\right]$,
for which (\ref{eq:scaling_matrix_form-1}) has a solution. The optimal
$\alpha_{0}$ then determines the maximal power threshold which can be provided
by the circuit and determines the circuit overload rate in some sense.
Finding such $\alpha_{0}$ corresponds to an optimization task
\begin{equation}
    \begin{aligned}
        \alpha_{0} & = \arg\max_{\alpha,\boldsymbol{\varphi}}\alpha\\
       & \text{subject to }  \\
          A\boldsymbol{\varphi}(\alpha)     & =\boldsymbol{b}(\boldsymbol{\varphi}(\alpha),\alpha)\\
         \alpha & \in\left[0,1\right] %\\        
    \end{aligned}
\label{eq:optimization-task}
% & \boldsymbol{\varphi} & \in \mathbb{R}^{n}.\nonumber 
\end{equation}

%This approach ensures the existence of solution $\boldsymbol{\varphi}(\widehat{\alpha})$
%for optimal $\widehat{\alpha}$ even if power demands $\boldsymbol{P}$
%are too high, assuming only a connected circuit with at least one
%voltage source. 
The optimal $\alpha_{0}$ smaller than
one gives us the information that the overhead wire is overloaded
and determines the rate of this overload.

\section{Theorerical analysis}\label{sec:theoreticalAnalysis}

%For theoretical justifying of the proposed approach, we need to make four 
The theoretical analysis justifying the proposed approach needs to make four
assumptions. The first one reads:
\begin{itemize}
  \item[(A1)] If the system \eqref{eq:DCPF_resistor_nodes}--\eqref{eq:DCPF_voltage_sources} has a solution for some $\bm P$, then it has a solution for all $\alpha\bm P$ with $\alpha\in[0,1]$.
\end{itemize}
This assumption says that if the system can be satisfied for a power demand $\bm P$, it can be satisfied for all smaller power demands. The second assumption requires:
\begin{itemize}
  \item[(A2)] If $\varphi_i$ are fixed for all $i\in J\cup I_0$, then \eqref{eq:DCPF_resistor_nodes} has a unique solution $\varphi_i$ with $i\in I$.
\end{itemize}
Since the subnetwork $I$ has no power demands, this assumption says that if there are prescribed values of the potential outside of this subnetwork, then the potentials inside the subnetwork will be distributed in a unique way. It is always true, since we have knowledge also about currents through the power sources/loads, and thus, the unique values of potentials $\varphi_i, i\in I$ can be obtained by Kirchhoff's laws. 

For the last two assumptions, consider the set of equations
\begin{equation}
f(\boldsymbol{\varphi}(\boldsymbol{\alpha}),\boldsymbol{\alpha})=A\boldsymbol{\varphi}(\boldsymbol{\alpha})-\boldsymbol{b}(\boldsymbol{\varphi}(\boldsymbol{\alpha}),\boldsymbol{\alpha})=A\boldsymbol{\varphi}(\boldsymbol{\alpha})-\left[\begin{array}{c}
\boldsymbol{0}\\
\frac{\boldsymbol{\alpha}\circ\boldsymbol{P}}{\boldsymbol{\varphi}(\boldsymbol{\alpha})_{J}}\\
\boldsymbol{U}
\end{array}\right]=0, \label{eq:problem_multi}
\end{equation}
to allow reducing the power demands in an unequal manner. Notation $\bm \alpha\circ \bm P$ denotes the Hadamard (component-wise) product of two vectors, and the fraction of two vectors is considered also in component-wise manner.
 Define the solution mapping of \eqref{eq:problem_multi} by
$$
  S(\bm \alpha) := \{\bm \vphi\mid \bm \vphi \text{ solves \eqref{eq:problem_multi} for }\bm \alpha\}
$$
and its domain by
$$
  \mathcal D := \{\bm\alpha \mid S(\bm\alpha)\text{ is nonempty }\} \cap [0,1]^{|\bm\alpha|}
$$
Then we impose the third assumption:
\begin{itemize}
  \item[(A3)] There exists some $M$ such that $\norm{S(\bm\alpha)}\le M$ for all $\bm\alpha\in \mathcal D$.
\end{itemize}
This says that the electric potential cannot be infinite. %All these three assumptions makes sense from the physical point of view. 
The last assumption states
\begin{itemize}
  \item[(A4)] The solution mapping $S$ is continuously differentiable on the interior of $\mathcal D$.
\end{itemize}
Note that due to Assumption (A1), the domain of the solution mapping is star-shaped with centre at $\bm 0$ and the differentiability is correctly defined. This assumption is also natural and makes sense from physical point of view. To summarize, all assumptions (A1)--(A4) are well justified and always satisfied in real-world electrical systems. %Since the left-hand side of \eqref{eq:problem_multi} is continuously differentiable, then the implicit function theorem implies that the solution mapping $S$ is continuously differentiable around some $\bm\alpha$ whenever the Jacobian of the left-hand side is regular at $\bm\vphi(\bm\alpha)$.

%To prove the convergence of our scheme, 
We start with the following lemma.

\begin{lemma}\label{lemma:bounded}
  Let Assumption (A3) be satisfied. Then the domain $\mathcal D$ is a closed set.
\end{lemma}
\begin{proof}
  Consider any sequence $\bm\alpha_k\to\bm \alpha$ such that $\bm\alpha_k\in\mathcal D$ for all $k$. We need to show that $\bm\alpha\in\mathcal D$. There are some corresponding $\bm\varphi_k\in S(\bm\alpha_k)$. Due to Assumption (A3), they are bounded and we may select a convergent subsequence, denoted without loss of generality by the same indices, $\bm \vphi_k \to\bm \vphi$. Since both sides of \eqref{eq:problem_multi} are continuous, we have $\bm\vphi\in S(\bm\alpha)$, which implies $\bm\alpha \in\mathcal D$.
\end{proof}

Now we are able to prove the following theorem.

\begin{theorem}\label{theorem:main}
  Let Assumptions (A1)--(A4) be satisfied. Then there are two possibilities:
  \begin{itemize}
    \item System \eqref{eq:scaling_matrix_form-1} has a solution for all $\alpha\in[0,1]$.
    \item System \eqref{eq:scaling_matrix_form-1} has a solution for all $\alpha\in[0,\alpha_0]$ but no solution for any $\alpha\in(\alpha_0,1]$ for some $\alpha_0\in(0,1)$. In such a case, the Jacobian $\nabla_\vphi f(\bm\vphi(\alpha), \alpha)$ is regular for all $\alpha\in[0,\alpha_0)$ but $\nabla_\vphi f(\bm\vphi(\alpha_0), \alpha_0)$ is singular.
  \end{itemize}
\end{theorem}
\begin{proof}
  We realize that due to Assumption (A1), either the system \eqref{eq:scaling_matrix_form-1} has a solution for all $\alpha\in[0,1]$ or there is some $\alpha_0$ such that the system has a solution on $[0,\alpha_0)$ and no solution on $(\alpha_0,1]$. Lemma \ref{lemma:bounded} implies that it will have a solution even for $\alpha_0$ in the second case.

  It remains to show in which cases the Jacobians are regular. First, we write \eqref{eq:problem_multi} in a more compact form
  \begin{equation}\label{eq:problem_compact1}
    \begin{pmatrix} A_{II} & A_{IJ} & A_{II_0} \\ A_{JI} & A_{JJ} & A_{JI_0} \\ 0 & 0 & E \end{pmatrix} \begin{pmatrix} \bm \vphi_I \\ \bm \vphi_J \\ \bm \vphi_{I_0} \end{pmatrix} = \begin{pmatrix} \bm 0 \\ \frac{1}{\bm\vphi_J}\left(\bm \alpha\circ\bm P \right) \\ \bm U \end{pmatrix},
  \end{equation}
  where $E$ stands for the identity matrix. The matrix on the left-hand side does not depend on the potentials $\bm\vphi$. Due to Assumption (A2), the square matrix $A_{II}$ is regular and we have
  $$
    \bm\vphi_I = A_{II}^{-1}\left(- A_{IJ}\bm\vphi_J - A_{II_0}\bm\vphi_{I_0}\right).
  $$
  Plugging this back to \eqref{eq:problem_compact1} yields
  $$
    A_{JJ}\bm\vphi_J + A_{JI}A_{II}^{-1}\left(- A_{IJ}\bm\vphi_J - A_{II_0}\bm \vphi_{I_0}\right) + A_{JI_0}\bm \vphi_{I_0} = \frac{1}{\bm\vphi_J} \left(\bm\alpha\circ\bm P \right)
  $$
  and in a simpler form
  \begin{equation}\label{eq:problem_compact2}
    g(\bm \vphi) := \bm\vphi_J\circ\left(A_{JJ} -  A_{JI}A_{II}^{-1}A_{IJ}\right) \bm\vphi_J  -  \bm\vphi_J\circ \left(A_{JI}A_{II}^{-1}A_{II_0}\bm U - A_{JI_0}\bm U \right) = \bm \alpha\circ\bm P.
  \end{equation}
 
  Now we perturb the right hand side of \eqref{eq:problem_compact2} by some $\bm r$ to get
  \begin{equation}\label{eq:problem_pert1}
    \bm\vphi_J\circ\left(A_{JJ} -  A_{JI}A_{II}^{-1}A_{IJ}\right) \bm\vphi_J  -  \bm\vphi_J\circ \left(A_{JI}A_{II}^{-1}A_{II_0}\bm U - A_{JI_0}\bm U \right) = \bm \alpha\circ\bm P + \bm r.
  \end{equation}
  We consider $\bm\alpha$ as a fixed parameter and perturb only $\bm r$. We will show that the solution mapping $\tilde S: \bm r\mapsto \bm\vphi_J$ is continuously differentiable around $\bm 0$. Equality \eqref{eq:problem_pert1} amounts to
  $$
    \bm\vphi_J\circ\left(A_{JJ} -  A_{JI}A_{II}^{-1}A_{IJ}\right) \bm\vphi_J  -  \bm\vphi_J\circ \left(A_{JI}A_{II}^{-1}A_{II_0}\bm U - A_{JI_0}\bm U \right) = (\bm \alpha + \frac{\bm r}{\bm P})\circ\bm P.
  $$
  This, due to the same reasons as above, is equivalent to 
$f(\boldsymbol{\varphi}(\widetilde{\boldsymbol{\alpha}}),\widetilde{\boldsymbol{\alpha}})=0$ with $\widetilde\alpha_i=\alpha_i+\frac{r_i}{P_i},i\in J$.
  
  From \cite[Theorem 1C.3]{dontchev2009implicit} we obtain that $\tilde S$ is continuously differentiable at $\bm 0$ if and only if $\nabla_\vphi g(\bm\vphi(\bm\alpha))$, defined in \eqref{eq:problem_compact2}, is regular. Due to the equivalence of \eqref{eq:problem_multi} and \eqref{eq:problem_compact2} and the equivalence of \eqref{eq:problem_pert1} and $f(\boldsymbol{\varphi}(\widetilde{\boldsymbol{\alpha}}),\widetilde{\boldsymbol{\alpha}})=0$, we realize that $\tilde S$ is continuously differentiable at $\bm 0$ if and only if $S$ is continuously differentiable at $\bm\alpha$. Combining these two facts we obtain that $S$ is continuously differentiable at $\bm\alpha$ if and only if $\nabla_\vphi g(\bm\vphi(\bm\alpha))$ is regular. Assumption (A4) states that $S$ is continuously differentiable on $[0,\alpha_0)$. Since $\alpha_0$ lies on the boundary of the domain, $S$ cannot be differentiable there. This whole paragraph implies that $\nabla_\vphi g(\bm\vphi(\bm\alpha))$ is regular for all $\alpha\in[0,\alpha_0)$ but $\nabla_\vphi g(\bm\vphi(\bm\alpha_0))$ is singular.
  
  The regularity of $\nabla_\vphi g$ is equivalent to the regularity of the Jacobian of
  $$
    \hat g(\bm\vphi_J) := \left(A_{JJ} -  A_{JI}A_{II}^{-1}A_{IJ}\right) \bm\vphi_J  -  \left(A_{JI}A_{II}^{-1}A_{II_0}\bm U - A_{JI_0}\bm U \right) - \frac{1}{\bm\vphi_J}(\bm \alpha\circ\bm P).
  $$
  We have
  $$
    \nabla_\vphi \hat g(\bm\vphi_J(\bm\alpha)) = A_{JJ} -  A_{JI}A_{II}^{-1}A_{IJ} + \diag\left(\frac{1}{\bm\vphi_J^2}(\bm \alpha\circ\bm P)\right),
  $$  
  where $\diag(\cdot)$ makes a diagonal matrix from a vector and the $\bm\vphi_J^2$ is understood component-wise. Define now the function from \eqref{eq:problem_compact1} by
  $$
    \tilde g(\bm\vphi) := \begin{pmatrix} A_{II} & A_{IJ} & A_{II_0} \\ A_{JI} & A_{JJ} & A_{JI_0} \\ 0 & 0 & E \end{pmatrix} \begin{pmatrix} \bm \vphi_I \\ \bm \vphi_J \\ \bm \vphi_{I_0} \end{pmatrix} - \begin{pmatrix} \bm 0 \\ \frac{1}{\bm\vphi_J}\left(\bm \alpha\circ\bm P \right) \\ \bm U \end{pmatrix}.
  $$
  Then
  $$
    \nabla_\vphi \tilde g(\bm\vphi(\bm\alpha)) = \begin{pmatrix} A_{II} & A_{IJ} & A_{II_0} \\ A_{JI} & A_{JJ} + \diag (\frac{1}{\bm\vphi_J^2}(\bm \alpha\circ\bm P))& A_{JI_0} \\ 0 & 0 & E \end{pmatrix} 
  $$
  and therefore
  $$
    \aligned
    \det \nabla_\vphi \tilde g(\bm\vphi(\bm \alpha)) &= \det \begin{pmatrix} A_{II} & A_{IJ} \\ A_{JI} & A_{JJ} + \diag (\frac{1}{\bm\vphi_J^2}(\bm \alpha\circ\bm P))\end{pmatrix} \\
    &= \det A_{II} \det \left(A_{JJ} + \diag \left(\frac{1}{\bm\vphi_J^2}(\bm\alpha\circ \bm\vphi)\right) - A_{JI}A_{II}^{-1}A_{IJ}\right),
    \endaligned
  $$
  where the second equality follows from the theory of Schur's complement. Assumption (A2) says that $\det A_{II}\neq 0$ and therefore, $\det \nabla_\vphi \tilde g(\bm\vphi(\bm\alpha))\neq 0$ if and only if $\det \nabla_\vphi g(\bm\vphi(\bm\alpha))\neq 0$. But this means that $\nabla_\vphi \tilde g(\bm\vphi(\bm\alpha))$ is regular if and only if $\nabla_\vphi g(\bm\vphi(\bm\alpha))$ is regular. Since $f(\bm\vphi(\alpha),\alpha) = g(\bm\vphi(\bm \alpha))$, the end of the previous paragraph implies the theorem statement.
\end{proof}

Assuming (A1)--(A4), Theorem \ref{theorem:main} ensures that the optimization task \eqref{eq:optimization-task} has nice properties which we will utilize to suggest a solution algorithm in the next section and to further investigate and numerically demonstrate in Section \ref{sec:Numerical-results}.%is well defined and there exists the unique optimal solution $\alpha_0$.

\section{Solution Procedure\label{sec:Proposed-Algorithm}}

%In previous sections, we have shown, that the optimization task  \eqref{eq:optimization-task} is well defined. Since there are some non-linear constraint, \eqref{eq:optimization-task} can be solved using various non-linear programming solvers. We seek a method with simple implementation and low computation requirements, since our aim is to incorporate a solution method into the open-source traffic simulator Eclipse SUMO.

In this section, we propose a method with simple implementation and low computation requirements. Our method is based on Theorem 2, which states that there is some $\alpha_0$ such that system \eqref{eq:scaling_matrix_form-1} is solvable for all $\alpha\in[0,\alpha_0]$ but not solvable for any larger $\alpha$. Therefore, we start with $\widetilde{\alpha}=0$ and incrementally increase the value of $\widetilde{\alpha}$ by some constant gain $\Delta\alpha$. After each update of the scaling parameter, the Newton-Raphson method is used to solve \eqref{eq:scaling_matrix_form-1}, with the $k$-th iteration evaluated as
\begin{align}
\boldsymbol{\varphi}_{\mathrm{k+1}} & =\boldsymbol{\varphi}_{\mathrm{k}}-\left(\nabla_{\boldsymbol{\varphi}}f\left(\boldsymbol{\varphi}_{\mathrm{k}},\widetilde{\alpha}\right)\right)^{-1}f\left(\boldsymbol{\varphi}_{k},\widetilde{\alpha}\right)\label{eq:Newton-update}\\
 & =\boldsymbol{\varphi}_{\mathrm{k}}-\left(A-\nabla_{\boldsymbol{\varphi}}\boldsymbol{b}\left(\boldsymbol{\varphi}_{\mathrm{k}},\widetilde{\alpha}\right)\right)^{-1}\left(A\boldsymbol{\varphi}_{\mathrm{k}}-\boldsymbol{b}\left(\boldsymbol{\varphi}_{\mathrm{k}},\widetilde{\alpha}\right)\right).\nonumber 
\end{align}

Theorem 2 also states that $\nabla_{\boldsymbol{\varphi}}f\left(\boldsymbol{\varphi}_{\mathrm{k}},\widetilde{\alpha}\right)$ converges to a singular matrix as $\widetilde\alpha \to \alpha_0$. Therefore, if we observe that the determinant of this matrix (which is the same as the one in \eqref{eq:Newton-update}) goes to zero, we imply that we are close to the optimal scaling parameter $\alpha_0$. If this happens, the previous value of $\widetilde{\alpha}$ is declared as the optimal value with  tolerance equal to $\Delta\alpha$. This procedure is summarized in Algorithm \ref{alg:Basic-algorithm}.

\begin{algorithm}[!ht]
\begin{algorithmic}[1]

\STATE  Set constant gain $\Delta\alpha$ to a proper value %(e.g. $10^{-4}$),
and $\widetilde{\alpha}=0$

\STATE  Solve \eqref{eq:scaling_matrix_form-1} for the fixed  $\widetilde{\alpha}$ using the NR method \label{alg:basic:NR method}

\STATE  If the NR method converges and $\widetilde{\alpha}<1$, set $\widetilde{\alpha}=\widetilde{\alpha}+\Delta\alpha$ and  GO TO line \ref{alg:basic:NR method}.

\STATE  Else $\widetilde{\alpha}-\Delta\alpha$ is declared as the optimal value with the tolerance $\Delta\alpha$ and with the corresponding solution $\boldsymbol\varphi (\widetilde{\alpha}-\Delta\alpha)$

\end{algorithmic}

\caption{Basic solution strategy\label{alg:Basic-algorithm}}
\end{algorithm}

The suggested procedure is convergent under the assumptions of Theorem 2. However, it has a significant drawback in the case when the original DC PF problem \eqref{eq:DCPF_matrix_form} is solvable (i.e. there exists solution of \eqref{eq:scaling_matrix_form-1} for $\alpha=1$). Then, all discretized values of $\alpha$ need to be passed before the optimal $\alpha_0=1$ is encountered. For this reason, we propose an efficient solution procedure in the next subsection.

\subsection{Efficient Solution Algorithm}

This efficient procedure replaces the incremental increase of $\widetilde\alpha$ by a variant of a bisection method. The complete proposed pseudo-code is shown in Algorithm~\ref{alg:Proposed-algorithm}
and its C++ implementation can be found on official Eclipse SUMO GitHub repository\footnote{\url{https://github.com/eclipse/sumo/blob/master/src/utils/traction_wire/Circuit.cpp}}.

\begin{algorithm}
\begin{algorithmic}[1]

\STATE  Initialize $\boldsymbol{\varphi} := \boldsymbol{\varphi}_0$, the scaling parameter $\widetilde\alpha:=1$, 
the best found solution $\widehat{\alpha}:=0$, and an empty buffer of non-admissible parameter values $S_{\alpha}:=[\,]$.

\STATE  Set tolerances $\Delta_\text{con}:=10^{-8}$, $\Delta_\text{opt}:=10^{-5}$,
$\Delta_\text{act}$:= $10^{-2}$, the maximal number of NR iterations $m_\text{NR}:=10$ and coefficient of bisection $\mathrm{c}_\text{bi}:=0.5$.

\WHILE{true}

\FOR {$\mathrm{iter} = 1,\dots,m_\text{NR}$}

\STATE  $\boldsymbol{\varphi}:=\boldsymbol{\varphi}-\left(\nabla_{\boldsymbol{\varphi}}f\left(\boldsymbol{\varphi},\widetilde\alpha\right)\right)^{-1}f(\boldsymbol{\varphi},\widetilde\alpha)$
%\begin{lyxgreyedout}
%~~~//NR update
%\end{lyxgreyedout}
\COMMENT{NR update}

\IF[NR converged] {$\left\Vert f(\boldsymbol{\varphi},\widetilde\alpha)\right\Vert <\Delta_\text{con}$ 
%\begin{lyxgreyedout}
%~//NR converged
%\end{lyxgreyedout}
%\COMMENT{NR converged}
}

%\STATE  optionally if not meet current or voltage limits go to \texttt{line} \ref{alg:state:stack_no_solution.push_back()} \label{alg:state:optional-current-voltage-limits}

\STATE  $\widehat{\alpha}:=\widetilde\alpha$, $\widehat{\boldsymbol{\varphi}}:=\boldsymbol{\varphi}$,

\STATE  \textbf{break}

\ELSIF[NR failed to converge] {$\mathrm{iter} = m_\text{NR}$
%\begin{lyxgreyedout}
%~//NR failed to converge
%\end{lyxgreyedout}
%\COMMENT{NR failed to converge}
}

\STATE  $S_{\alpha}[\text{end}+1]:=\widetilde\alpha$ \label{alg:state:stack_no_solution.push_back()}
\COMMENT{append $\widetilde\alpha$ to the buffer}

\ENDIF

\ENDFOR

\IF {$S_{\alpha}$ is empty}

\STATE  \textbf{return} $\widehat{\alpha}$, $\widehat{\boldsymbol{\varphi}}$
%\begin{lyxgreyedout}
%~~~//end of algorithm with optimal $\widehat{\alpha}$ and corresponded ${\boldsymbol{\varphi}(\widehat\alpha)}$ 
%\end{lyxgreyedout}
\COMMENT{end of algorithm with optimal $\widehat{\alpha}$ and corresponding ${\boldsymbol{\varphi}(\widehat\alpha)}$}
\ENDIF

\IF {$\left\Vert \widehat{\alpha}-S_{\alpha}[\text{end}]\right\Vert \geq \Delta_\text{act}$}

\STATE $\widetilde\alpha:=\widehat{\alpha}+\mathrm{c}_\text{bi}\cdot(S_{\alpha}[\text{end}]-\widehat{\alpha})$ 
%\begin{lyxgreyedout}
%~~~//bisection method
%\end{lyxgreyedout}
\COMMENT{bisection at $\mathrm{c}_\text{bi}$}
\ELSE[moving towards ill-conditioned problem]

\STATE $m_\text{NR} := 2\cdot m_\text{NR}$ \label{alg:state:max_NR_iter_adaptivity}
%\begin{lyxgreyedout}
%~//progressive increasing of maximal number of NR iters
%\end{lyxgreyedout}
\COMMENT{progressive increase of maximal number of NR iters}

\STATE $\Delta_\text{act}$ := $\Delta_\text{act}$/10 \label{alg:state:tolerance_adaptivity}
%\begin{lyxgreyedout}
%~//progressive decreasing of optimality tolerance
%\end{lyxgreyedout}
\COMMENT{progressive decrease of optimality tolerance}

\IF {$\Delta_\text{act} < \Delta_\text{opt}$}

\STATE \textbf{return} $\widehat{\alpha}$, $\widehat{\boldsymbol{\varphi}}$
%\begin{lyxgreyedout}
%~~~//end of algorithm with optimal $\widehat{\alpha}$ and corresponding ${\boldsymbol{\varphi}(\widehat\alpha)}$ 
%\end{lyxgreyedout}
\COMMENT{end of algorithm with optimal $\widehat{\alpha}$ and corresponding ${\boldsymbol{\varphi}(\widehat\alpha)}$}
\ENDIF

\STATE  $\widetilde\alpha:=$ $S_{\alpha}[\text{end}]$

\STATE  $S_{\alpha}:=S_{\alpha}[1:\text{end}-1]$
\COMMENT{remove the last element of $S_{\alpha}$ buffer}

\ENDIF

\ENDWHILE

\end{algorithmic}

\caption{Algorithm for solving the DC PF problem\label{alg:Proposed-algorithm}}
\end{algorithm}

The \texttt{while} loop determines the optimal value of the scaling parameter. It starts with $\widetilde\alpha=1$ and makes use of the NR method to solve \eqref{eq:scaling_matrix_form-1}. The NR updates are inside the \texttt{for} loop. If the NR method succeeds (within a tolerance $\Delta_\text{con}$), the lower bound $\widehat\alpha$ is updated. In the opposite case the buffer of failed candidate values $S_\alpha$ is extended by the current scaling parameter $\widetilde\alpha$. If the buffer is empty (which happens if the NR method converged for $\widetilde\alpha=1$), we found a solution of \eqref{eq:scaling_matrix_form-1} and the algorithm terminates.

If the buffer is not empty, it means that there are some scaling parameters for which the solution does not exist. In most cases, we determine the new $\widetilde\alpha$ as an interpolation between $\widehat\alpha$ and the smallest value of the buffer $S_\alpha$. However, if these two values are close to each other (measured by the prescribed tolerance $\Delta_\text{act}$), the NR may require more iterations to converge since the problem is ill-conditioned due to Theorem \ref{theorem:main}. In such a case, we double the allowed number of iterations for the NR algorithm. This is depicted in the last few lines of Algorithm \ref{alg:basic:NR method}.

Since the convergence of the NR method also depends on the initial estimate, we initialize it by assuming constant term $\boldsymbol{b}\left(U,1\right)$ in \eqref{eq:scaling_matrix_form-1}, where $U$ is the known nominal voltage of the traction substation, and then by solving the linear equation $A\boldsymbol{\varphi}_0=\boldsymbol{b}\left(U,1\right)$ for $\boldsymbol{\varphi}_0$. In the numerical implementation, we also incorporate current limits of traction substations (over-current protection) and voltage limits of the network in a simple way. However, for the sake of simplicity we will ignore them here.

The proposed algorithm is simple to implement and its computation requirements are lower than those required by classical approaches such as
interior-point methods. Furthermore, the initialization of $\widetilde{\alpha}=1$ is strongly beneficial as the proposed algorithm solves the original unlimited DC PF problem within one iteration. Then the solution is found quickly without any additional
computational requirements.

\section{Simulation and Numerical Results \label{sec:Numerical-results}}

For numerical validation, we consider toy examples as well as a simulation of real-world traffic.

\subsection{Toy examples}

The two toy test cases of the DC PF problem are motivated by a real trolleybus network. The first test case (Figure \ref{fig:scheme-test-case-1})
contains four vehicles with defined power demands (\SI{260}{\kilo\watt}, \SI{20}{\kilo\watt}, \SI{30}{\kilo\watt} and \SI{-5}{\kilo\watt} due to regenerative braking) running under overhead wire section connected to the traction substation represented by voltage source of \SI{600}{\volt}
using one connection point. The second test case (Figure~\ref{fig:scheme-test-case-2})
contains ten vehicles with uniform power demands of \SI{250}{\kilo\watt} and two connection points to the traction substation on the voltage level \SI{600}{\volt}. The mean distance between adjacent (neighboring) vehicles is in the order of hundreds of metres in both test cases; this corresponds to the mean resistance of \SIrange{0.023}{0.23}{\ohm} of conductor wire between voltage nodes.

\begin{figure}
\begin{centering}
\begin{minipage}[t]{0.39\columnwidth}%
\subfloat[\label{fig:scheme-test-case-1}]{\begin{centering}
\includegraphics[scale=0.50]{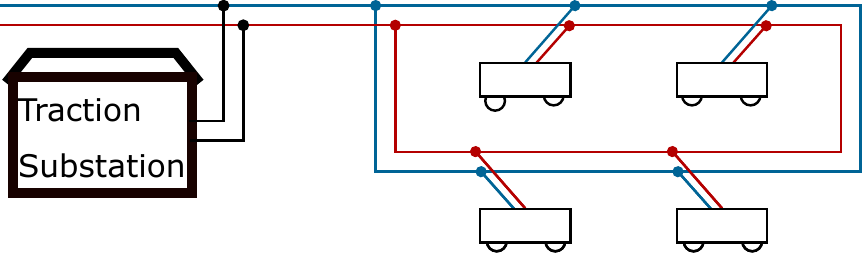}
\par\end{centering}
}%
\end{minipage}%
\begin{minipage}[t]{0.60\columnwidth}%
\subfloat[\label{fig:scheme-test-case-2}]{\begin{centering}
\includegraphics[scale=0.50]{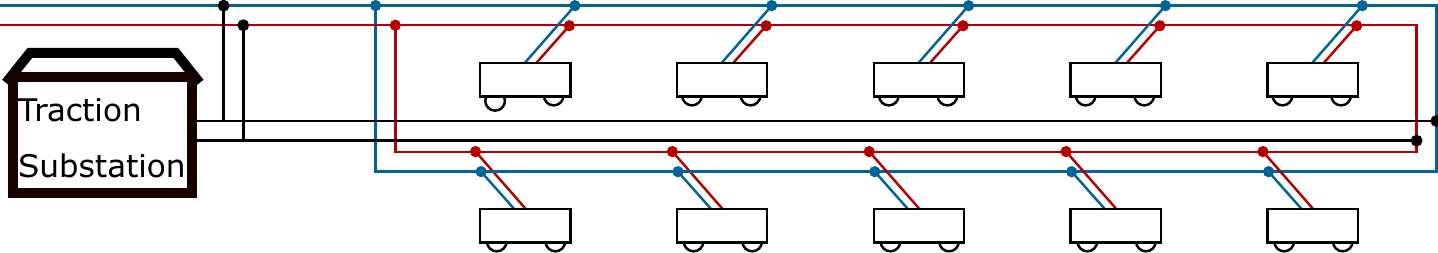}
\par\end{centering}
}%
\end{minipage}
\par\end{centering}
\begin{centering}
\caption{The scheme of the two toy test cases used for numerical analysis.}
\par\end{centering}
\end{figure}

\begin{figure}
\includegraphics[width=1\textwidth]{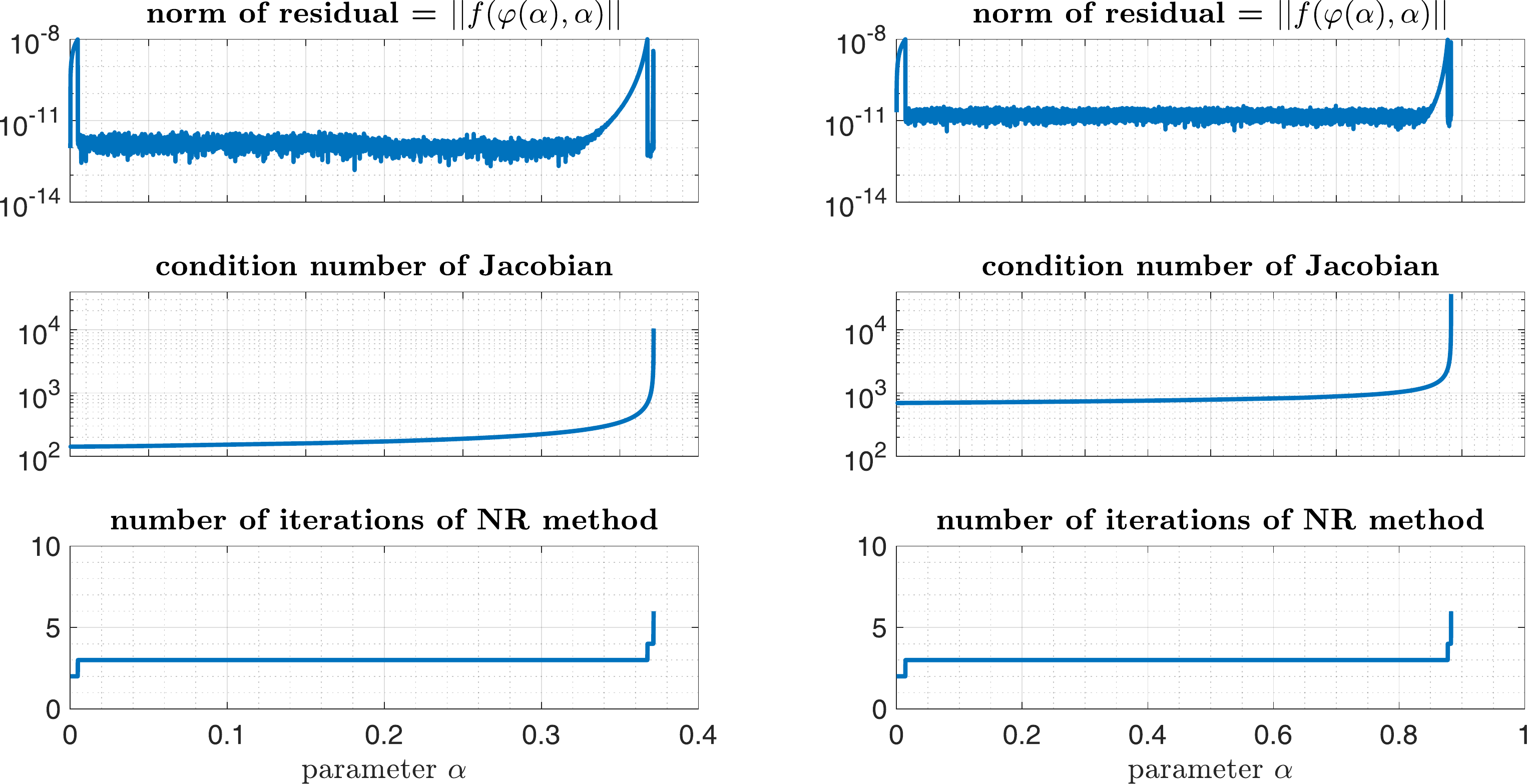}
\caption{Dependency of the condition number of Jacobian on
the proposed scaling parameter for the first toy case (left) and the second toy case (right).\label{fig:condition-test-cases-all}}
\end{figure}

%\begin{figure}
%\subfloat[\label{fig:condition-test-case-1}]{\includegraphics[width=1\textwidth]{../matlab_vypocetObvodu/conditioness_case_1}
%
%}
%
%\subfloat[\label{fig:condition-test-case-2}]{\includegraphics[width=1\textwidth]{../matlab_vypocetObvodu/conditioness_case_2}
%}
%
%\caption{dependency of the determinant and condition number of Jacobian on
%the proposed scaling parameter.}
%\end{figure}

Figure \ref{fig:condition-test-cases-all} empirically confirms the statements of Lemma 1 and Theorem 2. The left and right column of the figure show results for both toy cases introduced above. The top row shows the residuum of the solution which is always smaller than the threshold $\Delta_\text{con}=10^{-8}$. The middle row shows the condition number of the Jacobian $\nabla_{\boldsymbol{\varphi}}f\left(\boldsymbol{\varphi}(\alpha),\alpha\right)$ and the bottom row the number of iterations needed by the Newton-Raphson method. We see that this number is small but it is rising up as the condition number of Jacobian increases. The steep increase of the value of the condition number of Jacobian $\nabla_{\boldsymbol{\varphi}}f\left(\boldsymbol{\varphi}(\alpha),\alpha\right)$ in the middle row is apparent close to the critical value $\alpha_0$. It exactly corresponds to the theoretical conclusion of Theorem 2 stating that the interval of regularity of Jacobian matrix $\nabla_{\boldsymbol{\varphi}}f\left(\boldsymbol{\varphi}(\alpha),\alpha\right)$ is bounded by the critical value $\alpha_0$.  Finally, we can see that the results of Lemma 1 are also numerically confirmed as the solution exists only on some interval.

Let us now compare Matlab implementation of our Algorithm~\ref{alg:Proposed-algorithm} with the standard Matlab built-in function \texttt{fmincon} for solving
constrained non-linear optimization problems using the interior-point method by default.
%validation with global optimal search method
%The optimal scaling parameter has been found by global search approach
%as $\widehat{\alpha}=0.3715$ for the first test case and $\widehat{\alpha}=0.8849$
%for the second test case. These optimal values are successfully found
%by the proposed method (except tolerance $\Delta_{opt}$) as same
%as by \texttt{fmincon} function using interior-point method. Setting
%of initial and tolerance values (but with slightly different meaning
%of optimality tolerance value) is the same as in Algorithm \ref{alg:Proposed-algorithm} for all methods.
The optimal scaling parameters $\alpha_0=0.3715$ for the first test case and $\alpha_0=0.8849$ were successfully found both by Algorithm \ref{alg:Proposed-algorithm} and by \texttt{fmincon}. 
The time requirements are compared in Table~\ref{tab:Time-requirements}.
Calculations were performed in Matlab on Intel Core i5 proccesor
and the time requirements are average over 1000 evaluations. Our Algorithm \ref{alg:Proposed-algorithm} gives the solution about one order of magnitude
faster than \texttt{fmincon}. If we provide user-defined
gradient to \texttt{fmincon}, the time requirements are
still approximately six times higher than for Algorithm \ref{alg:Proposed-algorithm}.

\begin{table}
%\subfloat[]{\noindent \begin{centering}
%\begin{tabular}{|l|l|l|l|}
%\hline 
%\textbf{Test case 1} & Time {[}s{]} & Max t {[}s{]} & Min t {[}s{]} \tabularnewline
%\hline 
%\hline 
%fmincon & 0.0266 & 0.0484 & 0.0221 \tabularnewline
%\hline 
%fmincon with gradients & 0.0217 & 0.0412 & 0.0177 \tabularnewline
%\hline 
%proposed Algorithm \ref{alg:Proposed-algorithm} & 0.0032 & 0.0059 & 0.0027 \tabularnewline
%\hline 
%\end{tabular}
%\par\end{centering}
%
%}
%
%\subfloat[]{\noindent \begin{centering}
%\begin{tabular}{|l|l|l|l|}
%\hline 
%\textbf{Test case 2} & Time {[}s{]} & Max t {[}s{]} & Min t {[}s{]} \tabularnewline
%\hline 
%\hline 
%fmincon & 0.0403 & 0.0913 & 0.0355 \tabularnewline
%\hline 
%fmincon with gradients & 0.0261 & 0.0605 & 0.0219 \tabularnewline
%\hline 
%proposed Algorithm \ref{alg:Proposed-algorithm} & 0.0042 & 0.0096 & 0.0040 \tabularnewline
%\hline 
%\end{tabular}
%\par\end{centering}
%}

\begin{tabular}{|l|l|l|l|l|l|l|}
\hline
\multicolumn{1}{ |c|  }{\multirow{3}{*}{Algorithms} } & \multicolumn{3}{c|}{Test case 1} & \multicolumn{3}{c|}{Test case 2}  \tabularnewline
\cline{2-7} 
\multicolumn{1}{ |c|  }{} & \multicolumn{1}{c|}{mean t}  & \multicolumn{1}{c|}{max t} & \multicolumn{1}{c|}{min t} & \multicolumn{1}{c|}{mean t} & \multicolumn{1}{c|}{max t} & \multicolumn{1}{c|}{min t} \\
\multicolumn{1}{ |c|  }{} & \multicolumn{1}{c|}{[s]} & \multicolumn{1}{c|}{[s]} & \multicolumn{1}{c|}{[s]} & \multicolumn{1}{c|}{[s]} & \multicolumn{1}{c|}{[s]} & \multicolumn{1}{c|}{[s]}\\
\hline  
fmincon & 0.0266 & 0.0484 & 0.0221 & 0.0403 & 0.0913 & 0.0355 \tabularnewline
\hline 
fmincon with gradients & 0.0217 & 0.0412 & 0.0177 & 0.0261 & 0.0605 & 0.0219 \tabularnewline
\hline 
proposed Algorithm \ref{alg:Proposed-algorithm} & 0.0032 & 0.0059 & 0.0027 & 0.0042 & 0.0096 & 0.0040 \tabularnewline
\hline 
\end{tabular}

\caption{Comparison of mean, maximal and minimal time requirements of proposed Algorithm \ref{alg:Proposed-algorithm} and Matlab's \texttt{fmincon} for solving the DC PF problem for both test cases. \label{tab:Time-requirements}
% Number of iterations of the proposed method stands for the number of solving the DC PF problems \eqref{eq:scaling_matrix_form-1} with different value of parametr $\alpha$, the value in brackets is the total sum of all NR iterations.
}

\end{table}

\subsection{Simulation of a trolleybus network\label{sec:Numerical-results-simulation}}

In this final section, we present numerical simulations of a part of a real network. The simulations are performed in the Eclipse SUMO simulator. We show two test cases, in both we employ our Algorithm \ref{alg:basic:NR method} to solve the DC PF problem.

The first case is a simulation of a single trolleybus on a straight 8 kilometres long route with equally distant (\SI{200}{\metre}) intersections. The overhead wire network is powered by a traction substation at the beginning (position \SI{0}{\metre}) of the route. Figure \ref{fig:Simulation-of-single-vehicle} shows the results. In the top row, the actual value of parameter $\alpha$ is depicted. The value depends on the requested power (depicted in the bottom row), which is periodically oscillating as the trolleybus accelerates after passing each intersection and uses regenerative breaking before reaching the next one. The scaling parameter decreases (when accelerating) with the distance from the power source due to increased length and therefore also resistance of the overhead power line. This driving behaviour with the actual speed of the trolleybus is shown in the middle row. The last row shows the requested power and the received power. The latter is the requested power times the scaling constant. Note that the difference needs to be either covered by an additional source such as an on-board battery (this is the depicted case), or it results in slower acceleration.

\begin{figure}
\begin{centering}
\includegraphics[width=1\textwidth]{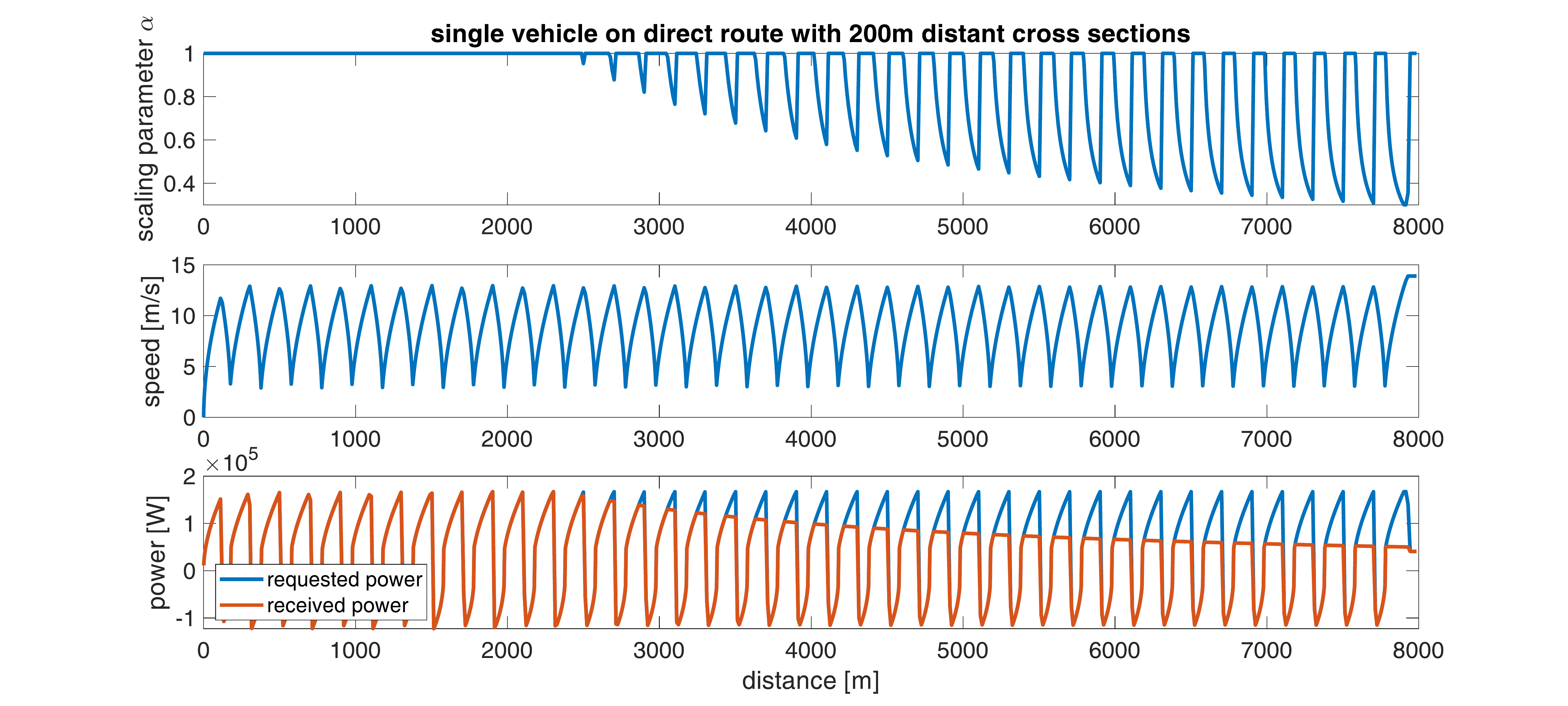}
\par\end{centering}
\caption{Simulation of single trolleybus of 8 kilometers on a straight route with equally distant (200m) cross sections.  \label{fig:Simulation-of-single-vehicle}}
\end{figure}

\begin{figure}

\begin{center}
\includegraphics[width=.49\textwidth]{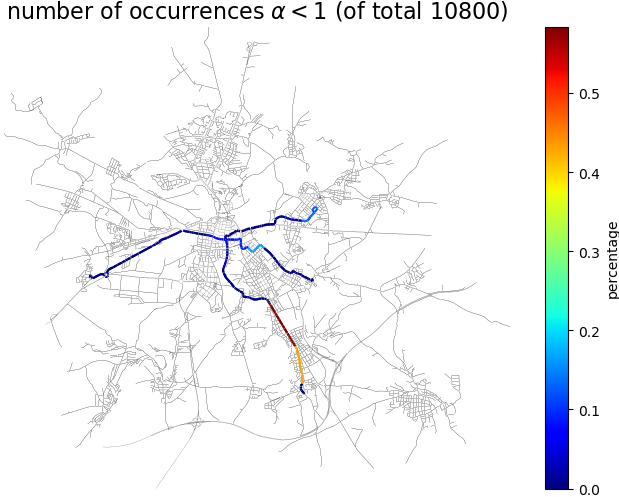}\hspace{.15cm}
\includegraphics[width=.49\textwidth]{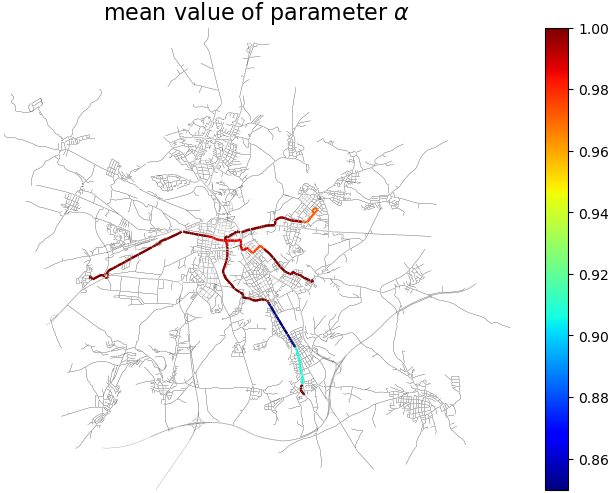}
\end{center}
\caption{Large-scale stress analysis of two trolleybus lines in the city of Pilsen, Czech Republic.\label{fig:pilsen_alpha}}

\end{figure}

The second case is a large-scale simulation of two hybrid-trolleybus lines in the city of Pilsen, Czech Republic. 
%The proposed method for solving the DC PF problem implemented in simulator Eclipse SUMO is used for the validation of an optimized trolleybus network in the city. This activity is pursued under smart mobility effort to increase the electrification of the vehicle fleet of municipal public transport operator. 
%
Figure \ref{fig:pilsen_alpha} shows the frequency of application of the proposed scaling parameter in simulations (it is smaller than one) and its mean value. 
The results indicate sufficient dimensions of overhead wires in the city with an exception in the south side of the city. In this location, mean value of parameter $\alpha$ is under 0.9 and the relative frequency of $\alpha<1$ is above \SI{40}{\percent}. This may indicate problems in the network operation and the necessity to add more sub-station connection points and/or overhead wire clamps into the circuit in order to improve its electrical parameters and decrease the demonstrated energy losses on the power line.

\section{Conclusion}

The formulation of the DC power flow problems leads to a system of non-linear equations that is not always solvable in its given form, despite the fact that
measurements and experiments on its real-world counterpart suggest that the physical system is able to reach a state for which a real solution exists.
To ensure solvability of the DC PF problem, we proposed to introduce a scaling parameter for power demands and presented
a fast algorithm to maximise its value while keeping the DC PF system solvable, together with theoretical and numerical verification.
The introduced strategy enables us to effectively find power demand
thresholds and thus to solve ill-defined DC PF problems.

The algorithm has been  demonstrated on representative toy cases. The performance of Matlab implementation of the proposed method has been compared to Matlab's fmincon; we demonstated that our approach outperforms the standard optimisation approach by a factor of 6 to 9. The viability of the algorithm and its C++ implementation build into the open-source traffic simulator Eclipse SUMO has been demonstrated on a real-world scenario. The C++ implementation is open-sourced and is provided to the community by authors as a part of Eclipse SUMO.
%\section*{Acknowledgements}

{\small
\bibliographystyle{panm20}
\bibliography{panm20PowerFlowProblemLiterature} 
}

% With BibTeX, uncomment the following three lines:
%\bibliographystyle{panm20} % bibliographic style - name of *.bst file to use
%\bibliography{panm20PowerFlowProblemLiterature}    % bibliographic database - name of *.bib file
%\bibliography{sample}

{\small
{\em Authors' addresses}:
{\em Jakub \v{S}ev\v{c}\'ik,  Jan P\v{r}ikryl,  V\'aclav \v{S}m\'idl}, University of West Bohemia, Research and Innovation Centre for Electrical Engineering, Pilsen, Czech Republic; 
 e-mail: \texttt{jsevcik@\allowbreak rice.zcu.cz, prikryl@\allowbreak rice.zcu.cz, vsmidl@\allowbreak rice.zcu.cz }.\\
 {\em Luk\'a\v{s} Adam}, Czech Technical University in Prague, Faculty of Electrical Engineering, Prague, Czech Republic; 
 e-mail: \texttt{lukas.adam.cr@\allowbreak gmail.com }.

}

\end{document}